\documentclass[10pt]{amsart}
\begin{document}
\title{Closures of $K$-orbits in the flag variety for $SU^*(2n)$}
\author{William M. McGovern}
\begin{abstract}
We characterize the $Sp_{2n}$-orbits in the flag variety for $SL_{2n}$ with rationally smooth closure via a pattern avoidance criterion, also showing that the singular and rationally singular loci of such orbit closures coincide.
\end{abstract}
\maketitle

\section{Introduction}
Let $G$ be a complex reductive group with Borel subgroup $B$ and let $K=G^\theta$ be the fixed point subgroup of an involution of $G$.  In this paper we continue the program begun in \cite{M07} and continued in \cite{MT08}, using pattern avoidance to characterize the $K$-orbits in $G/B$ with rationally smooth closure (as in \cite{LS90}).  Here we consider the case $G = SL(2n,\mathbb C), K = Sp(2n,\mathbb C)$.  We will adapt the techniques used in \cite{B98} to study Schubert varieties for complex classical groups, focussing on the poset and graph structures of the set of orbits with closures contained in a given one.

I would like to thank Axel Hultman for many very helpful email messages and Michael Kiyo for pointing out an error in an earlier version of this paper.

\section{Preliminaries}
Set $G = SL(2n,\mathbb C), K = Sp(2n,\mathbb C)$.  Let $B$ be the subgroup of upper triangular matrices in $G$.  The quotient $G/B$ may be identified with the variety of complete flags $V_0\subset V_1\subset\cdots\subset V_{2n}$ in $\mathbb C^{2n}$.  The group $K$ acts on this variety with finitely many orbits; these are parametrized by the set $I_{2n}$ of involutions in the symmetric group $S_{2n}$ without fixed points \cite{MO88,RS90}.  In more detail, let $\langle\cdot,\cdot\rangle$ be the standard nondegenerate skew form on $\mathbb C^{2n}$ with isometry group $K$.  Then a flag $V_0\subset\cdots\subset V_{2n}$ lies in the orbit $\mathcal O_\pi$ corresponding to the involution $\pi$ if and only if the rank of $\langle\cdot,\cdot\rangle$ on $V_i\times V_j$ equals the cardinality
$\#\{k: 1\le k\le i, \pi(k)\le j\}$ for all $1\le i,j\le 2n$.  

We will be using a modified version of the usual notion of pattern avoidance for permutations.  We say that $\pi=\pi_1\ldots\pi_{2n}$ (in one-line notation) includes the pattern $\mu=\mu_1\ldots\mu_{2m}$ if there are indices $i_1< i_2 <\cdots< i_{2m}$ permuted by $\pi$ such that $\pi_{i_j}>\pi_{i_k}$ if and only if $\mu_j>\mu_k$ (the usual definition would omit the condition that $\pi$ permute the $i_j$).  We say that $\pi$ avoids $\mu$ if it does not include it.  For example, the involution 47513826 includes the pattern 351624:  the indices $1,2,4,6,7,8$ are permuted by the involution and the first and third, fourth and sixth, and second and fifth of these are flipped.  On the other hand, the involution 65872143 does not include the pattern 2143, for although the indices $2,1,4,3$ occur in that order in the involution they are not permuted by it.

There are well-known poset- and graph-theoretic criteria for rational smoothness of complex Schubert varieties due to Carrell and Peterson \cite{C94}.  These have been extended by Hultman to our setting (or more generally to $K$-orbits in any flag variety $G/B$ where the symmetric pair $(G,K)$ corresponds to a real form $G_0$ of the reductive group $G$ whose Cartan subgroups form a single conjugacy class \cite{H09}).  To state them we first recall that the standard partial order on $K$-orbits, by containment of their closures, corresponds to the (restriction of the) reverse Bruhat order on $I_{2n}$ \cite{RS90}.  The poset $I_{2n}$ equipped with this order is then graded via the rank function
$$
r(\pi) =n^2- \sum_{i<\pi(i)} (\pi(i) - i - \#\{k\in\mathbb N: i<k<\pi(i), \pi(k)<i\})
$$
\noindent where this quantity equals the difference in dimension between $\mathcal O_\pi$ and $\mathcal O_c$, the unique closed orbit, corresponding to the involution $w_0=2n\ldots1$ \cite{RS90}.  Let $I_\pi$ be the interval consisting of all $\pi'\le\pi$ in the reverse Bruhat order.  Then Hultman has shown that $\bar{\mathcal O}_\pi$ is rationally smooth if and only if $I_\pi$ is rank-symmetric in the sense of having the same number of elements of rank $i$ as of rank $r-i$ for all $i$, where $r$ is the rank of $\pi$; equivalently, if and only if the rank generating function $P_\pi = \sum_{\pi'\le\pi} q^{r(\pi')}$ is palindromic \cite[5.9]{H09}.  If we make $I_\pi$ into a graph $BG_\pi$ by decreeing that the vertices $\mu$ and $\nu$ are adjacent if and only if $\nu = t\mu t\ne\mu$ for some transposition $t$ in $S_{2n}$, then $\bar{\mathcal O}_\pi$ is rationally smooth if and only $BG_\pi$ is regular of degree $r$.  If $\mu<\pi$ and we make the reverse Bruhat interval $[\mu,\pi]$ into a graph $BG_{\mu,\pi}$ by the same recipe, then $\bar{\mathcal O}_\pi$ is rationally smooth at $\mathcal O_\mu$ if and only if the degree of $\mu$ in $BG_{\mu,\pi}$ is $r(\pi) - r(\mu)$ \cite[4.5,5.8,6.7]{H09} (but in general $BG_{\mu,\pi}$ need not be regular or rank-symmetric in this situation).  In general the degree of $\mu$ in $BG_{\mu,\pi}$ is always at least $r:=r(\pi) - r(\mu)$.  We call $\mu$ an irregular vertex if it has larger degree than $r$.

\section{Main result}
We begin with a lemma about the inductive behavior of vertex degrees in Bruhat graphs.

\newtheorem*{lemma}{Lemma}
\begin{lemma}
Let $\mu,\pi$ be two involutions in $I_{2n}$ with $\mu\le\pi$ in reverse Bruhat order.  Let $t$ be a transposition of two indices flipped by both $\mu$ and $\pi$ and set $\pi=\tilde\pi t, \mu=\tilde\mu t$ (so that $\tilde\pi,\tilde\mu$ are {\bf not} in $I_{2n}$).  Let $\pi',\mu'$ be the unique involutions in $I_{2n-2}$ such that $\tilde\pi,\tilde\mu$ include the patterns $\pi',\mu'$, respectively, in the indices fixed by $t$.  Assume that the vertex $\mu'$ is irregular in $BG_{\mu',\pi'}$.  Then $\mu$ is irregular in $BG_{\mu,\pi}$.
\end{lemma}

\begin{proof}
Note first that the one-line notation of $\pi'$, for example, is obtained from that of $\pi$ by deleting the indices flipped by $t$ and then replacing the $i$th smallest of the surviving indices by $i$.  Thus if $\pi = 361542$ and $t$ flips 1 and 3, then $\pi' = 4321$.  We say that the transposition $(a,d)$ flipping the indices $a$ and $d$ with $a<d$ encapsulates the transposition $(b,c)$ with $b<c$ if $a<b<c<d$.  Then the rank difference $r(\pi) - r(\mu)$ is given by $r(\pi') - r(\mu') + 2(n(\mu) - n(\pi))$, where $n(\mu),n(\pi)$ are the numbers of transpositions in $\mu,\pi$, respectively, encapsulating $t$.  Now every edge from $\mu'$ in $BG_{\mu',\pi'}$ corresponds to an edge from $\mu$ in $BG_{\mu,\pi}$ in an obvious way.   For every transposition counted by $n(\mu)$ but not $n(\pi)$ one easily locates two additional edges from $\mu$ in $BG_{\mu,\pi}$, showing that $\mu$ is irregular whenever $\mu'$ is, as desired.
\end{proof}

Now we can characterize the $K$-orbits with rationally smooth closure.

\newtheorem{theorem}{Theorem}
\begin{theorem}
The orbit $\mathcal O_\pi$ has rationally smooth closure if and only if $\pi$ avoids the 17 patterns
$351624,64827153,57681324,53281764,43218765,65872143,21654387,\linebreak 21563487,34127856,43217856,34128765,36154287,21754836,63287154,54821763,\linebreak 46513287,21768435$.
\end{theorem}

\begin{proof}
Note first that this list of bad patterns is stable under the automorphism of the Dynkin diagram:  the first nine patterns are fixed by this automorphism while the next four pairs of patterns are interchanged.  Suppose first that $\pi$ coincides with one of the bad patterns.  Then the bottom vertex in $BG_\pi$ is irregular, as one sees from the following table.  Here the rank of each bad pattern (regarded as an element of $I_6$ or $I_8$) is given in the middle column and the transpositions labelling the edges from the bottom vertex are given in the right column, abbreviating the flip of the $i$th and $j$th coordinates by $ij$.
\vskip .3in
\begin{tabular}{| c | c | c |}\hline
vertex & rank & edges from bottom vertex\\
\hline
351624 & 4 & 12,13,14,23,24\\
\hline
64827153 & 5 & 12,13,23,24,25,34,35\\ \hline
57681324 & 5 & 12,13,14,23,24,34\\ \hline
53281764 & 7 & 12,13,14,23,24,25,34,35\\ \hline
43218765 & 8 & 12,13,14,15,23,24,25,26,34,35\\ \hline
65872143 & 4 & 12,13,23,24,34\\ \hline
21654387 & 10 & 12,13,14,15,16,17,23,24,25,26,34,35\\ \hline
21563487 & 11 & 12,13,14,15,16,17,23,24,25,26,34,35\\ \hline
34127856 & 10 & 12,13,14,15,16,23,24,25,26,34,35\\ \hline
43217856 & 9 & 12,13,14,15,23,24,25,26,34,35\\ \hline
34128765 & 9 & 12,13,14,15,16,23,24,25,26,34,35\\ \hline
36154287 & 9 & 12,13,14,15,16,23,24,25,26,34,35\\ \hline
21754836 & 9 & 12,13,14,15,16,23,24,25,26,34,35\\ \hline
63287154 & 6 & 12,13,23,24,25,34,35\\ \hline
54821763 & 6 & 12,13,23,24,25,34,35\\ \hline
46513287 & 8 & 12,13,14,15,23,24,25,34,35\\ \hline
21768435 & 8 & 12,13,14,15,23,24,25,34,35\\  \hline
\end{tabular}
\vskip .3in

 Now if $\pi$ contains a bad pattern, then repeated use of Lemma 1 shows that $BG_{\mu,\pi}$ is irregular at $\mu$, where the one-line notation of $\mu$ is obtained from that of $\pi$ by rewriting the indices in the bad pattern in decreasing order and leaving the other indices unchanged.
Conversely, suppose that $\pi$ avoids all patterns in the above list.  We will show that the rank generating polynomial $P_\pi$ is palindromic, or more precisely it is the product of sums of the form
$1+q+\cdots +q^t$ for various $t$.  Let $\pi = \pi_1\ldots\pi_{2n}$ and assume first that $2n - \pi_1\le\pi_{2n} - 1$ (i.e., 1 is closer to the end of $\pi_1\ldots\pi_{2n}$ than $2n$ is to the beginning).  Set $\pi^{(1)} = t\pi t$, where $t$ is the transposition interchanging $\pi_1$ and $\pi_1 + 1$, so that 1 appears one place further to the right in $\pi^{(1)}$ than in $\pi$.  Define $\pi^{(2)},\ldots,\pi^{(2n-\pi_1)}$ similarly, so that 1 appears at the end of $\pi^{(2n-\pi_1)}$.  If $\mu=\mu_1,\ldots\mu_{2n} < \pi$ then Proctor's criterion for the Bruhat order \cite{P82} shows that $ \mu_1\ge\pi_1$.  If $\mu_1 = \pi_1$, then one checks that $\mu' < \pi'$, where $\mu',\pi'$ are obtained from $\mu,\pi$ by omitting the indices 1 and $\mu_1$, replacing all indices $i$ between 1 and $\mu_1$ by $i-1$, and replacing all indices $j>\mu_1$ by $j-2$; moreover, $\pi'$ continues to avoid all bad patterns.  If instead $\mu_1 > \pi_1$, then we claim that $\mu\le\pi^{(1)}$ and that $\pi^{(1)}$ continues to avoid all bad patterns. If this holds, then induction shows that $\mu\le\pi^{(\mu_1 - \pi_1)}$ whence we may as above eliminate the indices 1 and $\mu_1$ from $\mu$ and $\pi^{(\mu_1 - \pi_1)}$ and repeat the above procedure. Using the formula for the rank function in $I_{2n}$, we deduce that $P_\pi$ factors in the way claimed above, where the first factor is $1 + q + \cdots + q^{2n-\pi_1}$.  

To prove the claim, set $\pi_1 = k, \pi_{k+1} = i$ and suppose that there is $\mu$ with $\mu<\pi,\mu\not\le\pi^{(1)}$, and $\mu_1>\pi_1$.  There are two cases, according as $i<k$ or $i>k+1$.  If $i<k$ then we look at the indices greater than $k$ among $\pi_1,\ldots,\pi_k$.  If these do not occur in order of increasing size, then the pattern $p:=465132$ is included in $\pi$, in such a way that the 4 corresponds to $\pi_1  = k$.  The assumption $2n - \pi_1\le\pi_{2n} - 1$ implies that $\pi\ne p$, so that $\pi$ is the product of three transpositions forming the pattern $p$ and at least one more transposition.  Now one checks that no matter how one chooses this transposition to force $2n- \pi_1\le\pi_{2n} - 1$ we get a bad pattern in $\pi$, a contradiction; more precisely, one of the five patterns 46513287, 63287154, 65872143, 64827153, or 57681324 must occur in $\pi$.  If the indices greater than $k$ do occur in order of increasing size, then (since $\pi$ is an involution) the indices less than $k$ {\sl not} occurring among $\pi_1,\ldots,\pi_k$ are all larger than $i$, whence the indices $2,\ldots,i-1$ occur among $\pi_1,\ldots,\pi_{k-1}$ (and $\pi_k=1$).  These conditions are incompatible with $\mu<\pi$ and $\mu\not\le\pi^{(1)}$, so this case leads to a contradiction.  So we must have $i>k+1$.  Now if $\pi_j>k$ for any $j<i$ then one of the patterns 351624 or 361542 must occur in $\pi$:  the former is ruled out since it is a bad pattern and the latter, combined with the condition that $2n-\pi_1\le\pi_{2n} - 1$, would force one of the bad patterns 36154287 or 53281764 to occur in $\pi$ (arguing as in the case above where $\pi$ includes the pattern $p$).  So   $\pi_1, \ldots\pi_k$ must be a permutation of $1,\ldots,k$, whence $k$ must be even.  Now the absence of the patterns 43218765 and 43217856 in $\pi$ implies that $i=k+2$.  In this case the only way that we can have $\mu<\pi,\mu\not\le\pi^{(1)}$ is if exactly $k-1$ indices less than $k$ appear among $\mu_1,\ldots,\mu_k$, which is a contradiction since $\mu$ has no fixed points. 

If instead $\pi_{2n} - 1 < 2n - \pi_1$, then one repeats the above argument, replacing 1 by $2n$ and moving $2n$ to the left instead of 1 to the right.  Thus we define $\pi^{(1)},\pi^{(2)}$, and so on, so that $2n$ appears one place to the left in $\pi^{(1)}$ than it does in $\pi$; if $\mu\le\pi$ then we must have $\mu_{2n}\le\pi_{2n}$, and if $\mu_{2n}<\pi_{2n}$, then we must have $\mu\le\pi^{(1)}$, lest $\pi$ contain a bad pattern.  Here the two \lq\lq bad seeds" that must be ruled out are 546213 and 532614; these give rise to the bad patterns 21768435, 54821763, 65872143, 64827153, 57681324, 21754836, and 53281764.

Finally, we must ensure in both cases that $\pi^{(1)}$ avoids all bad patterns whenever $\pi$ does.   This requires that we rule out four more \lq\lq bad seeds", namely 216543, 432165, 215634, and 341265; we achieve this by ruling out the bad patterns 21654387, 43218765, 34127856, 43217856, 21563487, and 34128765.  Excluding also the bad pattern 351624 of length 6, we see that if $\pi$ avoids all bad patterns then $P_\pi$ factors in the desired way and $\mathcal O_\pi$ has rationally smooth closure, as required. 
\end{proof}

\section{Smoothness and the bottom vertex}
We now consider reverse Bruhat intervals $[\mu,\pi]$ and their graphs $BG_{\mu,\pi}$. We will find (as for Schubert varieties in type $A$) that it is only necessary to test one vertex in this graph to determine whether or not $\bar{\mathcal O}_\pi$ is (rationally) smooth at $\bar{\mathcal O}_\mu$.  

\newtheorem*{result}{Theorem 2}
\begin{result}
If $\mu<\pi$ and the degree of $\mu$ in $BG_{\mu,\pi}$ equals $r(\pi) - r(\mu)$, then $\bar{\mathcal O}_\pi$ is smooth along $\bar{\mathcal O}$. In particular, the singular and rationally singular loci of $\bar{\mathcal O}$ coincide.
\end{result}

\begin{proof}
Assume first that $\mathcal O_\mu = \mathcal O_c$, the closed orbit.  Fix a basis $(e_i)$ of $\mathbb C^{2n}$ such that

\[
 \langle e_i,e_j\rangle = \begin{cases} 1 & \text{if $i<j,i+j = 2n+1$}\\
-1 & \text{if $i>j,i+j = 2n+1$}\\
0 & \text{otherwise}
\end{cases}
\]

\noindent where $\langle\cdot,\cdot\rangle$ is the skew form.  Let $(a_{ij})$ be a family of complex parameters indexed by ordered pairs $(i,j)$ satisfying either $i\le n<j$ or $n<i<j$.  We assume that $a_{ij} = -a_{2n+1-j,2n+1-i}$  and $a_{i,2n+1-i} = 0$ if $i\le n$ but otherwise put no restrictions on the $a_{ij}$.  Define a basis $(b_i)$ of $\mathbb C^{2n}$ via 

\[ b_i = \begin{cases} e_i + \displaystyle\sum_{j=n+1}^{2n} a_{ij} e_j & \text{if $i\le n$}\\
e_i + \displaystyle\sum_{j=i+1}^{2n} a_{ij} e_j & \text{otherwise}
\end{cases}
\]

\noindent Then the Gram matrix $G := (g_{ij} = (\langle b_i,b_j\rangle))$ of the $b_i$ relative to the form satisfies
\[
 g_{ij} = \begin{cases} 2a_{i,2n+1-j} & \text{if $i<j\le n$}\\
-g_{ji} & \text{if $j<i\le n$}\\
a_{j,2n+1-i} & \text{if $i<n<j<2n+1-i$}\\
1 & \text{if $i<n<j= 2n+1-i$}\\
-g_{ji} & \text{if $j<n<i$}\\
0 & \text{otherwise}
\end{cases}
\]
\noindent Thus the matrix $G$ is skew-symmetric and has zeroes below the antidiagonal from lower left to upper right.  The antidiagonal entries are all $\pm1$.  Now one checks that the set $\mathcal F$ of all flags $V_0\subset\ldots\subset V_{2n}$ where $(b_i)$ runs through all bases obtained as above from the $a_{ij}$ and $V_i$ is the span of $b_1,\ldots b_i$ is a slice in the sense of Brion to $\mathcal O_c$ at the flag $f_c$ corresponding to the basis $(e_i)$ \cite[2.1]{Br99}.  Intersecting $\mathcal F$ with $\bar{\mathcal O}_\pi$ we get another slice to $\mathcal O_c$ at $f_c$.  

By hypothesis there are $n^2 - n - r(\pi)$ distinct conjugates $c=t w_0 t$ of $w_0$ by a transposition $t$ such that $c\not\le\pi$.  One computes that $d:=tw_0$ also satisfies $d\not\le\pi$.  Writing $d$ as $d_1\ldots d_{2n}$, let $i$ be the smallest index such that $\pi_1\ldots\pi_i\not\le d_1\ldots d_i$ in the standard partial order on sequences used to characterize the Bruhat order \cite{P82}.  Thus if $\pi_1\ldots\pi_i$ is rearranged in increasing order as $\pi'_i\ldots\pi'_i$ and similarly $d_1\ldots d_i$ is rearranged as $d'_1\ldots d'_i$, then $\pi'_j>d'_j$ for some $j$.  Then for some $k$ there are more indices $\ell\le i$ with $d_\ell<k$ than indices $m\le i$ with $\pi_m<k$.  Equating all minors of the appropriate size lying in the first $i$ rows and columns $d_1,\ldots ,d_i$ of the Gram matrix to 0, we arrive at $n^2 - n -r(\pi)$ polynomials vanishing on $\bar{\mathcal O}_\pi\cap\mathcal F$, each involving a distinct variable raised to the first power with coefficient $\pm1$.  Then the Jacobian matrix of these polynomials has rank $n^2 - n - r(\pi)$, whence by the Jacobian criterion both $\mathcal F\cap\bar{\mathcal O}_\pi$ and $\bar{\mathcal O}_\pi$ are smooth at $\mathcal O_c$, as desired \cite[2.1]{Br99}

If $\mathcal O_c$ is replaced by any orbit $\mathcal O_\mu\subset\bar{\mathcal O}_\pi$, then let $G_\mu$ be the matrix whose $ij$-entry is 1 if $j=\mu_i>i, -1$ if $j=\mu_i<i$, and 0 otherwise.  This is the Gram matrix of a basis $(b_i)$ obtained by suitably rearranging the basis $(e_i)$; let $f_\mu$ be the corresponding flag.  Now consider the set of all Gram matrices $G$ whose $ij$-entries agree with those of $G_\mu$ if $j\ge\mu_i$ and whose other possibly nonzero entries are determined as follows.  There are $n^2 - n - r(\mu)$ conjugates $c$ of $\mu$ by transpositions with $c>\mu$.  Write each $c$ as $c_1\ldots c_{2n}$ and let $i$ be the smallest index with $c_i<\mu_i$.  Then the other possibly nonzero entries of $G$ appear in the positions $(i,c_i)$ together with their transposes $(c_i,i)$.  Entries of $G$ not in one of the positions specified above are 0.  There are no further restrictions on these entries apart of course from being skew-symmetric.  This set of Gram matrices stands in bijection to a set $\mathcal F_\mu$ of flags which is a slice to $\bar{\mathcal O}_\mu$ at $f_\mu$.  Then one shows as above that if the degree of $\mu$ in $BG_{\mu,\pi}$ equals the difference $r(\pi) - r(\mu)$, then there are $n^2 -n - r(\pi)$ polynomials vanishing on $\bar{\mathcal O}_\pi\cap\mathcal F_\mu$ whose Jacobian matrix has rank $n^2 -n - r(\pi)$, whence again $\bar{\mathcal O}_\pi$ is smooth along $\mathcal O_\mu$, as desired.
\end{proof}

There are two other symmetric pairs $(G,K)$ of complex reductive groups satisfying the hypothesis of \cite{H09} (that all Cartan subgroups in the corresponding real form $G_0$ of $G$ are conjugate), namely $(Spin(2n,\mathbb C),Spin(2n-1,\mathbb C))$ and $(E_6,F_4)$.  In the first case all $K$-orbits in $G/B$ have smooth closure.  In the second case, eleven out of the forty-five $K$-orbits have rationally singular closure.  Hultman has checked in each case that the bottom vertex of the Bruhat graph detects the rational singularity.  It is not known whether smoothness and rational smoothness are equivalent for these orbit closures.

\end{document}